\documentclass[12pt,req]{amsart}

\usepackage{amscd}
\usepackage{times,fancyhdr}
\usepackage{epsfig}
\usepackage{graphicx}
\usepackage{color}
\usepackage{mathtools}
\usepackage{mathabx}
\usepackage{stmaryrd}
\usepackage{mathrsfs}
\usepackage{enumerate}

\usepackage{array}

\usepackage{tikz-cd}
\usepackage{graphicx}
\usepackage{tikz}
\usetikzlibrary{decorations.fractals}
\usetikzlibrary{lindenmayersystems}
\usepackage[matrix,arrow,ps]{xy}
 \UsePSspecials{dvips}
\usetikzlibrary{positioning}
\input{xy}
\usepackage{hyperref}
\usepackage{csquotes}
\usepackage{relsize}

\usetikzlibrary{positioning}
\usepackage{extarrows}

\usepackage{footnote}
\makesavenoteenv{tabular}
\makesavenoteenv{table}

\makesavenoteenv{tabular}
\makesavenoteenv{table}\topmargin 0in
\newcommand{\E}{\mathbb{E}}

\newcommand{\ad}{\operatorname{ad}}

\newcommand{\Z}{\mathbb{Z}}
\newcommand{\SL}{\operatorname{SL}}
\newcommand{\diag}{\operatorname{diag}}
\newcommand{\GL}{\operatorname{GL}}

\newtheorem{theorem}{Theorem}[section]
\newtheorem*{abeltheorem}{Theorem~\ref{abel}}
\newtheorem*{rk2theorem}{Theorem~\ref{T-rk2}}
\newtheorem{proposition}[theorem]{Proposition}
\newtheorem{lemma}[theorem]{Lemma}

\newtheorem{cor}[theorem]{Corollary}

\theoremstyle{definition}
\newtheorem{definition}[theorem]{Definition}

\def\diag{\operatorname{diag}}
\def\ad{\operatorname{ad}}
\def\Sp{\operatorname{Sp}}
\def\Aut{\operatorname{Aut}}
\def\Hom{\operatorname{Hom}}
\def\Spec{\operatorname{Spec}}

\def\frak g{{\frak g}}

\renewcommand\a{\alpha}                
\renewcommand\b{\beta}

\def\la{\langle}
\def\ra{\rangle}

\newcommand{\N}{{\mathbb N}}

\def\Q{{\mathbb C}}

\def\Q{{\mathbb Q}}
\def\R{{\mathbb R}}

\def\Z{{\mathbb Z}}

\def\Ad{\text{Ad}}

\def\frak{\mathfrak}

\def\la{\langle}
\def\ra{\rangle}

\DeclareMathOperator{\re}{re}
\DeclareMathOperator{\im}{im}

\DeclareMathOperator{\wts}{wts}
\DeclareMathOperator{\End}{End}
\DeclareMathOperator{\rank}{rank}

\begin{document}

\title{\bf{Strong integrality of inversion subgroups of \\Kac--Moody groups}}

\date{\today}

\author[A. Ali ]{Abid Ali}
\address{Department of Mathematics, Rutgers University, Piscataway, NJ 08854-8019, U.S.A.}
\email{abid.ali@rutgers.edu}

\author[L. Carbone]{Lisa Carbone}
\address{Department of Mathematics, Rutgers University, Piscataway, NJ 08854-8019, U.S.A.}
\email{carbonel@math.rutgers.edu}

\author[D. Liu]{Dongwen Liu}
\address{School of Mathematical Sciences, Zhejiang University, Hangzhou 310027, P.R. China}\email{maliu@zju.edu.cn}

\author[Scott H.\ Murray]{Scott H.\ Murray}
\address{Department of Mathematics, Rutgers University, Piscataway, NJ 08854-8019, U.S.A.}
\email{scotthmurray@gmail.com}

\thanks{2000 Mathematics subject classification. Primary 20G44, 81R10; Secondary 22F50, 17B67}

\begin{abstract}  Let $A$ be a symmetrizable generalized Cartan matrix with corresponding  Kac--Moody algebra $\frak{g}$ over $\Q$. Let $V=V^{\lambda}$ be an integrable highest weight $\frak{g}$-module and let $V_\Z=V^{\lambda}_\Z$ be a $\Z$-form of $V$. Let  $G$ be an associated minimal representation-theoretic Kac--Moody group
and let $G(\Z)$ be its integral subgroup. Let $\Gamma(\Z)$ be the Chevalley subgroup of $G$, that is, the subgroup
 that stabilizes the lattice $V_{\Z}$ in $V$.  For a subgroup $M$ of  $G$, we say that $M$ is integral if  $M\cap G(\Z)=M\cap \Gamma(\Z)$ and that $M$ is strongly integral if there exists $v\in V^{\lambda}_\Z$ such that, for all $g\in M$,  $g\cdot v\in V_{\mathbb{Z}}$ implies $g\in G({\mathbb{Z}})$. We prove strong integrality of  inversion subgroups $U_{(w)}$ of $G$ where, for $w\in W$, $U_{(w)}$ is the group generated by positive real root groups that are flipped to negative roots by $w^{-1}$. We use this to prove strong integrality of  subgroups of the unipotent subgroup $U$ of $G$  generated by commuting real root groups.  When $A$ has rank 2, this gives strong integrality of subgroups $U_1$ and $U_2$ where $U=U_{1}{\Large{*}}\  U_{2}$  and each $U_{i}$ is generated by `half' the positive real roots.
\end{abstract}

\thanks{The second author's research is partially supported by the Simons Foundation, Mathematics and Physical Sciences-Collaboration Grants for Mathematicians, Award Number 422182.  The third author's research is partially supported by NSF of Zhejiang Province LZ22A010006 and NSFC 12171421.}

\maketitle

\section{Introduction}\label{S-intro}
This paper concerns the generalization to infinite-dimensional symmetrizable Kac--Moody algebras of the construction of  Chevalley groups associated to finite-dimensional semi\-simple Lie algebras.

Let $A$ be a symmetrizable generalized Cartan matrix. Let $\frak{g}$ be a Kac--Moody algebra over $\Q$ associated with $A$. We denote simple roots by $\a_1,\dots ,\a_{\ell}$, the set of roots by $\Delta$, the set of positive (respectively negative) roots by $\Delta_{\pm}$ , 
and the set of real roots by $\Delta^{\re}$. If $A$ is not of finite type, then $\frak{g}$ is infinite dimensional. 
Let $V=V^{\lambda}$ be an integrable highest-weight $\frak{g}$-module  with highest weight $\lambda$, highest weight vector $v_\lambda$, and defining homomorphism $\rho:\frak g\to\End(V^{\lambda})$.
Let $\chi_{\pm\alpha_i}(t)\in\Aut(V)$ denote $\exp(t\rho(e_{i}))$  (resp.\ $\exp(t\rho(f_{i}))$),  for each $i\in\{1,\dots ,\ell\}$, where $t\in\Q$ and $e_{i}$ (resp.\ $f_i$) is the Chevalley--Serre generator corresponding to the simple root $\pm\a_i$.
A (split) \emph{minimal representation-theoretic Kac--Moody group} is 
$$G(\Q)=G^\lambda(\Q)=\langle \chi_{\alpha_i}(t),\ \chi_{-\alpha_i}(t)\mid i\in\{1,\dots ,\ell\},\ t\in\Q\rangle.$$

We let  ${\mathcal U}_{{\mathbb{Z}}}\subseteq {\mathcal U}$ denote the $\Z$-form of the universal enveloping algebra of $\mathfrak{g}$ (\cite{Ti87}; see also Section~\ref{S-KMring} below).  
As in \cite{CG}, we set $V_{\mathbb{Z}} = \mathcal{U}_{\mathbb{Z}}\cdot v_{\lambda}.$
The subgroup of $G(\Q)$ generated over $\Z$ is
\begin{align*}
G(\Z)&=\langle\chi_{\alpha_i}(t),\ \chi_{-\alpha_i}(t)\mid i\in\{1,\dots ,\ell\},\ t\in\Z\rangle 
\end{align*}
and the \emph{Chevalley subgroup} of $G(\Q)$ is the subgroup the stabilizes the integral lattice $V_\Z$:
 $$\Gamma(\Z)=\{g \in G(\Q) \mid g\cdot V_\Z= V_\Z\}.$$ 
 This work concerns the question of {\it integrality} of $G(\Q)$. By this we mean: 
 $$\text{`Is $G(\Z)=\Gamma(\Z)$?'}$$
 It is easy to show that $G(\Z)\subseteq \Gamma(\Z)$. Thus it remains to determine if $\Gamma(\Z)\subseteq G(\Z)$.

The answer to this question has not yet been established for any family of   Kac--Moody groups. Here we consider integrality of certain subgroups of $G(\Q)$. We call a subgroup $M$ of $G(\Q)$  \emph{integral} if $M\cap G(\Z)= M\cap \Gamma(\Z)$.
This is equivalent to showing that for all $g\in M$, $g\cdot V_\Z = V_\Z$ implies that $g\in G(\Z)$.
Note that $M_1\subseteq M_2$ and $M_2$ integral implies that $M_1$ is integral.
However  $M_1$ and $M_2$ integral \emph{does not} imply that the subgroup $\langle M_1,M_2\rangle$ generated by $M_1$ and $M_2$ is integral.

In the finite dimensional case, the property of integrality for semi-simple algebraic groups $G(\Q)$, with root system $\Delta$ and Lie algebra $\frak g$, was established in \cite{Chev1995}. Chevalley constructed  (what was his notion of) an affine group scheme associated to $G(\Q)$ and $V_\Z$ and showed that its coordinate algebra  is generated over $\Z$. This was a starting point for  Grothendieck and Demazure who organized the SGA3 seminar (\cite{SGA3}) where they developed the theory of  split reductive group schemes over arbitrary base schemes. Demazure and Grothendieck's  theory assumed known the classification  of reductive groups over an algebraically closed field due primarily to Chevalley, which also holds over $\Z$.

Let $L$ be a lattice satisfying $Q\subseteq L\subseteq P$ where  $Q$ is the root lattice and $P$ is the weight lattice of $\frak g$. 
Chevalley showed that ($\Delta$, $L$) defines an affine group scheme $G_\Z$ over $\Z$ such that
$$\Gamma(\Z)=G_\Z(\Z)= \Hom(\Spec(\Z),G_{\Z})$$
is its set of integral points (\cite{Chev1995}). The group $G_{\Z}(\Q)\cong G(\Q)$ is now known as a Chevalley group. If $L=Q$, then $G_{\Z}(\Q)$ is called {\it adjoint} and if $L=P$, then $G_\Z(\Q)$ is called {\it simply connected}.  It follows that the group $\Gamma(\Z)$ is  generated by the automorphisms 
$$\{\chi_{\a_i}(t), \chi_{-\a_i}(t)\mid \ t\in\Z\},$$
so if $G_\Z(\Q)$ is simply connected, it is integral.

To our knowledge, there is no proof in the literature that $G(\Z)=\Gamma(\Z)$ for  semi-simple linear algebraic groups $G(\Q)$ that doesn't use the language of group schemes. In certain cases, it is straightforward to verify integrality using properties of the linear algebraic group and underlying representation. For example, it is easy to see that $\SL_n(\Z)$ is the stabilizer of the standard lattice $\Z^n$ in $\Q^n$. See also \cite{Soule2007}, where Soul\'e showed that the  group $\E_7(\Z)$  coincides with  $\E_7(\R)\cap \Sp_{56}(\Z)$ where $\E_7(\R)$ is the non-compact real form and $\Sp_{56}(\Z)$ is the stabilizer of the standard lattice (and the canonical symplectic form) in the fundamental representation of dimension 56  of the Lie algebra $\frak e_7$.


The methods used to prove integrality in the finite dimensional case do not extend to Kac--Moody groups $G(\Q)$. Here we address the question of integrality of subgroups of $G(\Q)$ using only the definition of $G(\Q)$ acting on the integrable highest weight module $V$. We introduce the following slight modification of the notion of integrality.

\begin{definition}
Let  $M$ be a subgroup of the Kac--Moody group $G(\Q)$ and let $v_0\in V_\Z$. We say that $M$ is \emph{strongly integral with respect to $v_0\in V_\Z$} if, for all $g\in M$,  $g\cdot v_0\in V_{\mathbb{Z}}$ implies that $g\in G({\mathbb{Z}})$.
\end{definition}
This is stronger than integrality since it only requires us to test the action on a single vector $v_0$, rather than on the entire set~$V_{\mathbb{Z}}$. It is easy to see that strong integrality implies integrality (Lemma~\ref{Strong}).

The (positive) unipotent subgroup of $G(\Q)$ is
 $$U=U(\Q)=\langle \chi_{\alpha_i}(t)\mid i\in\{1,\dots ,\ell\},\ t\in\Q\rangle.$$
Unipotent subgroups play an important role in the study of the structure and representation theory of  Kac--Moody groups. They are constituents of group decompositions such as Iwasawa and Birkhoff decompositions, which provide 
important tools  for studying Kac--Moody groups and their applications.

In this paper, as a first step towards proving integrality of $U$, we prove strong integrality of inversion subgroups of $U$, which are defined as follows.  For $w$ in the Weyl group $W$, the {\it inversion subgroup} $U_{(w)}$ is given as 
 $$U_{(w)}=\langle U_{\beta}\mid \beta\in \Phi_{(w)}\rangle,$$ 
 where
$
 \Phi_{(w)}=\{\beta\in \Delta_{+}\mid w^{-1}\beta\in \Delta_{-}\}
$ and $U_\beta=\{ \chi_\beta(t)\mid t\in\Q\}$ for $\chi_\beta(t)$ as in Subsection~\ref{generators}.

We write $\widetilde{w}$ for the usual lift of $w$ to $G(\Q)$  (see Subsection~\ref{generators}).
Our main result is the following (Section~\ref{S-proof}).
\begin{theorem} (Strong integrality of inversion subgroups)\label{T-main}
Let $G(\Q)$ be as above. Let $U$ be the positive unipotent subgroup of $G(\Q)$. For $w\in W$,
$U_{(w)}$ is strongly integral with respect to $\widetilde{w}v_{\lambda}$.
\end{theorem}

If the matrix $A$ is of finite type (see Subsection~\ref{GCM}), then the group
$$G(\Q)=G^\lambda(\Q)=\langle \chi_{\alpha_i}(t),\ \chi_{-\alpha_i}(t)\mid i\in\{1,\dots ,\ell\},\ t\in\Q\rangle$$ is a semi-simple linear algebraic  group. Furthermore, if the set of weights of $V=V^\lambda$ contains all the fundamental weights, then $G(\Q)$ coincides with the  simply connected Chevalley group $G_\Z(\Q)$ of type $A$.

Our methods give a proof of integrality of unipotent subgroups of simply connected Chevalley groups $G_\Z(\Q)\cong G(\Q)$, using only the action of $G(\Q)$ on $V$. As a corollary of Theorem~\ref{T-main}, we have the following (Subsection~\ref{SS0-findim}).

 \begin{cor}\label{findimcor}
Suppose that the matrix $A$ has finite type. Suppose also that the set of weights of $V$ contains all the fundamental weights. Then  the positive unipotent subgroup $U$ of $G(\Q)$ is integral.
\end{cor}

 Theorem~\ref{T-main} also yields integrality of other subgroups of $U$.
In Subsection~\ref{SS-abel} we prove:
\begin{theorem}\label{abel}
If $M$ is a subgroup of $U$ generated by commuting real root subgroups, then
$M$ is integral.
\end{theorem}
In Section~\ref{SS-rank2}, we consider the case where $G$ has rank 2. In this case $U=U_{1}{\Large{*}}\,U_{2}$  where each $U_{i}$ is generated by `half' the positive real roots (see \cite{CKMS}).

\begin{theorem}\label{T-rk2} Let $G(\Q)$ be of rank 2 with positive unipotent subgroup $U=U_{1}{\Large{*}}\,U_{2}$ where $U_i$ are abelian if $A$ is symmetric, and nilpotent of class 2 if $A$ is not symmetric.   Then the groups $U_i$ are integral.
\end{theorem}

Though our methods do not currently extend to a proof of integrality for $U$, 
we conjecture{\footnote{An earlier paper by a subset of the authors (arXiv:1803.11204v2 [math.RT] Proposition 6.4) contains an error.}} that integrality holds for $U$. 
 It would then be straightforward to show that $G(\Z)=\Gamma(\Z)$ in the symmetrizable Kac--Moody case.

The authors would like to thank Shrawan Kumar for his interest in this work and for helpful discussions.

\section{Kac--Moody Algebras}\label{S-prelim}
In this section, we review some facts about Kac-Moody algebras and their representations.

\subsection{Generalized Cartan matrices}\label{GCM}
Let $I=\{1,2,\dots,\ell\}$ and let $A=(a_{ij})_{i,j\in I}$ be a {\it generalized Cartan matrix}. That is,
\begin{align*}
a_{ij}&\in{\Z};&  
a_{ii}&=2;\\
a_{ij}&\leq 0, \text{ for } i\neq j; \text{ and}&
a_{ij}& =0 \iff a_{ji}=0,
\end{align*}
for all $i,j\in I$.
The matrix $A$ is \emph{symmetrizable} if there exist positive rational numbers $q_1,\dots, q_{\ell}$, such that
the matrix  $\diag(q_1,\dots, q_\ell) A$ is symmetric.
We say that $A$ is of \emph{finite type} if $A$ is positive definite, and that $A$ is of {\it affine type} if $A$ is positive semidefinite but not positive definite. If $A$ is not of finite or affine type, we say that $A$ has \emph{indefinite type}.  
In particular, $A$ is of \emph{hyperbolic type} if $A$ is neither of finite nor affine type, but every proper, indecomposable submatrix is either of finite or of affine type. 

We  assume throughout that $A$ is a  symmetrizable generalized Cartan matrix.
\subsection{Generators and relations for $\frak g$} 
Let $\frak{h}$ be a $\Q$-vector space of dimension $2\ell-\rank(A)$, and let $\langle\circ,\circ\rangle: \frak{h}^{\ast}\times\frak{h}\to \Q$ denote the natural nondegenerate bilinear pairing between $\frak{h}$ and its dual. 
Fix \emph{simple roots} $\Pi=\{\alpha_1,\dots,\alpha_{\ell}\}\subseteq \frak{h}^{\ast}$ and \emph{simple coroots}
$\Pi^{\vee}=\{\alpha_1^{\vee},\dots,\alpha_{\ell}^{\vee}\} \subseteq
\frak{h}$ such that $\Pi$ and $\Pi^{\vee}$ are linearly independent, and $\langle\alpha_j,\alpha_i^{\vee}\rangle=\alpha_j(\alpha_i^{\vee})=a_{ij}$.

As in \cite{M1}, \cite{M2}, and \cite[Theorem 9.11]{Kac90}, the associated \emph{Kac--Moody algebra} $\frak{g}$ is the Lie algebra over $\Q$ with generating set  
$\frak{h}\cup\{e_i,f_i\mid i\in I\}$ and defining relations:
\begin{align*}
 [h,h']&=0;\\
[h,e_i]&=\langle\alpha_i,h\rangle e_i;&
[h,f_i]&=-\langle\alpha_i,h\rangle f_i;\\
[e_i,f_i]&=\alpha_i^{\vee};&
[e_i,f_j]&=0;\\
(\ad e_i)^{-a_{ij}+1}(e_j)&=0;&
(\ad f_i)^{-a_{ij}+1}(f_j)&=0;
\end{align*}
for $h,h'\in\frak{h}$, and $i,j\in I$ with $i\ne j$.

\subsection{Roots and the Weyl group}\label{RWG}
The \emph{roots}  of $\frak{g}$ are the nonzero  
$\a\in\frak{h}^{\ast}$ for which the corresponding root space
$$\frak{g}_{\alpha} := \{x\in \frak{g}\mid[h,x]=\alpha(h)x\text{ for all $h\in \frak{h}$}\}$$
is nontrivial. The simple roots $\a_i$ have root spaces $\frak{g}_{\a_i}=\Q e_i$.  
Every root $\alpha$ can be written in the form $\alpha=\sum_{i=1}^{\ell} k_i\alpha_i$ where the $k_i$ are integers, with either all $k_i \ge 0$, in which case $\alpha$ is called {\it positive}, or all $k_i\leq 0$, in which case $\alpha$ is called {\it negative}.  We denote the set of roots by $\Delta$ and the set of positive (resp.\ negative) roots by $\Delta_+$ (resp.\ $\Delta_-$). 
The \emph{root lattice} is $Q:=\mathbb{Z}\Pi\subseteq\frak h^*$ and the \emph{coroot lattice} is $Q^\vee:=\mathbb{Z}\Pi^\vee\subseteq\frak h$.

The Lie algebra  $\frak{g}$ has a triangular decomposition  (\cite[Theorem 1.2]{Kac90})
$$\frak{g}=\frak{n}^- \oplus \frak{h} \oplus \frak{n}^+,$$ 
where 
$$\frak{n}^+ =\bigoplus_{\alpha\in\Delta_+}\frak{g}_{\alpha}\quad\text{and}\quad
\frak{n}^- = \bigoplus_{\alpha\in\Delta_-}\frak{g}_{\alpha}.$$

Since $A$ is  symmetrizable,   $\frak{g}$ admits a nondegenerate symmetric invariant bilinear form $(\circ,\circ)$ satisfying 
\begin{align*}
  \left(\a^\vee_{i},\a^\vee_{j} \right) &= a_{ij}, &  \left({e}_{i},{f}_{j} \right) &= \delta_{ij}, \\
  \left(\mathfrak{h}, \mathfrak{n}_{+} \right)&=0,&
  \left(\mathfrak{h}, \mathfrak{n}_{-} \right)&=0.
\end{align*}

We define the \emph{simple reflection} $w_i: \frak h^\ast\to\frak h^\ast$ by
$$w_i(v)    =
   v - \la v,\alpha_i^{\vee}\ra\alpha_i.$$
The group $W\subseteq \GL(\frak{h}^{\ast})$ generated by the simple reflections is called the {\it Weyl group}. The induced bilinear form on ${\frak{h}^{\ast}}$ is $W$-invariant. 
The group $W$ comes equipped with a length function
$\ell\colon W\longrightarrow \mathbb{Z}_{\ge0}$
defined as $\ell(w)=k$, where $k$ is the smallest number such that $w$ is the product of $k$ simple root reflections.
Such a minimal length word $$w=w_{i_1}w_{i_2}\cdots w_{i_k}$$ is called a \emph{reduced word}.

A root $\alpha\in\Delta$ is called \emph{real} if there exists $w\in W$ such that $w\alpha$ is a
simple root. A root $\alpha$ which is not real is called 
\emph{imaginary}. 
We denote by $\Delta^{\re}$ the set of real roots and $\Delta^{\im}$ the
set of imaginary roots. We have $\Delta^{\re}=W\Pi$.  Similarly define sets $\Delta^{\re}_\pm :=  \Delta^{\re}\cap \Delta_\pm$
and  $\Delta^{\im}_\pm :=  \Delta^{\im}\cap \Delta_\pm$.
For $\a\in\Delta^{\re}$, the reflection in $\a$ is defined by
$$w_\a(v)    =   v - \la v,\alpha^{\vee}\ra\alpha,$$
where $\alpha^\vee$ is the usual coroot of $\alpha$ defined using the bilinear form $(\circ, \circ)$ on $\mathfrak{h}$.
Since $w\a=\a_i$ for some $w\in W$, $i\in I$, we have $w_\a =w w_i w^{-1}\in W$.


\subsection{Integrable highest weight modules}\label{repthy} 
The {\it weight lattice} $P \subseteq\frak{h}^{\ast}$ is the dual lattice of $Q^\vee$, that is,
$$P=\{\lambda\in \frak{h}^{\ast}\mid \la \lambda, \alpha_i^{\vee}\ra\in\Z,\ i\in I\}.$$
An element $\lambda\in P$ is {\it dominant} if $\la \lambda, \alpha_i^{\vee}\ra\geq 0$ for all $i\in I$, and is {\it regular} if $\la \lambda, \alpha_i^{\vee}\ra\ne 0$ for all $i\in I$. We denote the set of dominant elements of $P$ 
by $P_{+}$ and write $Q_+=\sum_{i=1}^\ell \Z_{\ge0}\a_i$.
The \emph{dominance order} on $\frak{h}^{\ast}$ is defined by $\lambda <\mu$ if $\mu-\lambda\in Q_+$.

Let $V$ be a $\frak{g}$-module with defining homomorphism 
$\rho:\frak g\to\End(V)$.
Then the \emph{weight space} of~$\mu\in \frak{h}^\ast$ is
$$V_\mu:= \{v\in V\mid \rho(h)(v) = \mu(h) v \text{ for all $h\in\frak{h}$} \},$$
and the set of \emph{weights} of $V$ is
$$\wts(V)=\left\{\mu\in \frak h^\ast \mid V_{\mu}\neq 0\right\}.$$
Recall also that $x\in\frak{g}$ acts \emph{locally nilpotently} on $V$ if, for each $v\in V$, there is a natural number $n$ such that
$\rho(x)^{n}(v)=0$. 
Then  $V$ is called an \emph{integrable} module if
it is a direct sum of its weight spaces, that is
$V=\bigoplus_{\mu\in \wts(V)} V_{\mu}$,
and each generator
$e_i$ or $f_i$ acts locally nilpotently on~$V$.


We call $\lambda\in\wts(V)$ the \emph{highest weight} of $V$ if it is largest element of $\wts(V)$ in the dominance order. 
If $V$ has a highest weight we call it a \emph{highest weight module}.
If $\mu$ is a weight of a highest weight module $V$, then $\mu=\lambda-\beta$ for some $\beta\in Q_+$. 
Among all modules with highest weight $\lambda\in\frak{h}^{\ast}$, there is a unique irreducible $\frak{g}$--module  (Prop. 9.3, \cite{Kac90}), which we denote by $V^{\lambda}$. 
The module $V^{\lambda}$ is integrable if and only if $\lambda\in P_+$, that is, $\lambda$ is a dominant  weight (Lemma 10.1, \cite{Kac90}). In this case, we get $\wts(V^\lambda)\subseteq P$. 

From now on, we take $V =V^\lambda$ to be an integrable highest weight $\mathfrak{g}$-module over $\mathbb{Q}$ for a fixed dominant, regular weight $\lambda$ and a fixed highest weight vector $v_{\lambda}$. For $x\in \mathfrak{g}$ and $v\in V$, we denote the action of $x$ on $v$ by $x\cdot v:=\rho(x)(v)$.

\section{Representation-theoretic Kac--Moody groups}\label{S-KMring}
Let $\frak{g}$ be the Kac--Moody algebra  over $\Q$ with symmetrizable generalized Cartan matrix $A$.
Let $V =V^\lambda$ to be an integrable highest weight $\mathfrak{g}$-module over $\mathbb{Q}$ for a fixed dominant, regular weight $\lambda$ and let $v_{\lambda}$ be a fixed highest weight vector .

The Chevalley involution $\omega$ is an automorphism of $\frak{g}$ with $\omega^2=1$, $\omega(h)=-h$ for all $h\in\frak h$, and $\omega(\frak{g}_\a)=\frak{g}_{-\a}$ for all $\a\in\Delta$.

\subsection{$\Z$--forms} 
Let ${\mathcal U}={\mathcal U}_{\Q}(\frak{g})$ be the universal enveloping algebra of $\frak{g}$. 
Let
\begin{itemize}
\item ${\mathcal U}^+_{{\mathbb{Z}}}$ be the ${\mathbb{Z}}$-subalgebra generated by 
$\dfrac{{e_i}^{m}}{m!}$ for $i\in I$
and $m\geq 0$;

\item \vspace{1mm}${\mathcal U}^-_{{\mathbb{Z}}}$ be the ${\mathbb{Z}}$-subalgebra generated by 
$\dfrac{{f_i}^{m}}{m!}$ for $i\in I$
and $m\geq 0$; and

\item \vspace{1mm}${\mathcal U}^0_{{\mathbb{Z}}}$ 
be the ${\mathbb{Z}}$-subalgebra generated by 
$$\left (\begin{matrix}
h \\ m\end{matrix}\right ):=\dfrac{h(h-1)\dots (h-m+1)}{m!},$$ for
$h\in Q^{\vee}$ and $m\geq 0$.
\end{itemize}
As in \cite{CG} and \cite{Ti81}, \cite{Ti87}, the \emph{$\Z$-form of $\mathcal U$} is ${\mathcal U}_{{\mathbb{Z}}}$, the ${\mathbb{Z}}$-subalgebra generated by ${\mathcal U}^+_{{\mathbb{Z}}}$, ${\mathcal U}^-_{{\mathbb{Z}}}$, and ${\mathcal U}^0_{{\mathbb{Z}}}$.
By \cite{Ti87}, we have $ \mathcal U^-_{ \Z}\otimes  \mathcal U_{ \Z}^0\otimes \mathcal U^+_{ \Z}\cong  \mathcal U_{ \Z}$.

We  define 
 $$\frak{g}_{\Z}\ =\ \frak{g}\cap {\mathcal U}_{\Z}\quad\text{and}\quad
 \frak{h}_{\Z}\ =\ \frak{h}\cap {\mathcal U}_{\Z}.$$ 
Note that $\frak{g}=\frak{g}_{\Z}\otimes_{\Z}\Q$  and $\frak{h}=\frak{h}_\Z\otimes_\Z\Q$.

Fix the highest weight vector $v_\lambda$ in $V$. 
We have
$\mathcal{U}^+_{\mathbb{Z}}\cdot v_{\lambda}=\mathbb{Z} v_{\lambda}$
since every $x_\alpha$ with $\alpha\in \Delta^{\re}_+$ annihilates $v_{\lambda}$. Also
$\mathcal{U}^0_{\mathbb{Z}}\cdot v_{\lambda}=\mathbb{Z} v_{\lambda}$
since ${\mathcal U}^0_{{\mathbb{Z}}}$ acts  on $v_{\lambda}$ as scalar multiplication by an integer. Therefore
$\mathcal{U}_{\mathbb{Z}}\cdot v_{\lambda}=\mathcal{U}^-_{\mathbb{Z}}\cdot v_{\lambda}.$
We define the \emph{$\Z$-form of $V$} to be
$$V_{\mathbb{Z}}  := \mathcal{U}_{\mathbb{Z}}\cdot v_{\lambda} = \mathcal{U}^-_{\mathbb{Z}}\cdot v_{\lambda}$$ and  note that
$V=V_{\mathbb{Z}}\otimes_\Z \mathbb{Q}.$
For $\mu\in \wts(V)$, set $V_{\mu, \mathbb{Z}}=V_{\mu}\cap V_{\mathbb{Z}}$.

\subsection{The  groups $G(\Q)$ and $G(\Z)$}\label{generators}
For $i\in I$ and $t\in \mathbb{Q}$  set 
$$\chi_{\alpha_{i}}(t) =\exp (t \rho(e_i))=\sum^\infty_{m=0}\frac{t^m \rho\left(e_i\right)^m}{m!},\quad 
\chi_{-\alpha_{i}}(t) =\exp (t \rho(f_i))=\sum^\infty_{m=0}\frac{t^m \rho\left(f_i\right)^m}{m!},
$$ 
where $\rho$ is the defining homomorphism for $V=V^{\lambda}$.
Since $V$ is integrable, the $\chi_{\alpha_{i}}(t)$ are elements of $\GL(V)$. 
We let $G(\mathbb{Q})\le \GL(V)$ be the group generated by $\chi_{\pm\alpha_{i}}(t)$ for $i\in I$ and $t\in \mathbb{Q}$.
We  refer to $G(\Q)=G^{\lambda}(\mathbb{Q})$ as a \emph{minimal representation-theoretic Kac--Moody group}.
Similarly we define $G(\mathbb{Z})$ as the subgroup of $G(\Q)$ generated by $\chi_{\pm\alpha_{i}}(t)$ for $i\in I$ and $t\in \mathbb{\Z}$.
We denote the action of $g\in G$ on $v\in V$ by $g\cdot v$.

For each $i\in I$ and $t\in \mathbb{Q}^\times$, we define
$$\widetilde{w}_{i}(t)= \chi_{\alpha_i}(t)\chi_{-\alpha_i}(-t^{-1})\chi_{\alpha_i}(t) $$ 
and set $\widetilde{w}_{i}=\widetilde{w}_{i}(1)$. 
Next set
$$h_{\alpha_i}(t)\ =\ \widetilde{w}_{i}(t) \widetilde{w}_{i}(1)^{-1}$$
for $t\in Q^{\times}$ and $i\in I$, and let $H\subseteq G$ be the subgroup generated by the $h_{\alpha_i}(t)$.
Let $N$ denote the subgroup of $G(\Q)$ generated by $H$ and the $\widetilde{w}_{i}$. By Lemma~4 in Section 5.4 of \cite{Ti87}, there exists a unique endomorphism
$N\longrightarrow W$
with kernel $H$.
For $w\in W$, take a reduced word $w=w_{i_1}w_{i_2}\cdots w_{i_k}$ and define
$$ \widetilde{w}:= \widetilde {w}_{i_1} \widetilde {w}_{i_2}\cdots \widetilde{w}_{i_k}.$$ 
Now $\widetilde{w}_{i}\in G(\mathbb{Z})\subseteq \Gamma(\Z)$ and so
\begin{eqnarray*}
\widetilde{w}\cdot V_{\mathbb{Z}}= V_{\mathbb{Z}}
\end{eqnarray*}
for all $w\in W$.

Note that the adjoint representation makes $\mathfrak{g}$ into a $G$-module and we also get $\Ad(\widetilde{w})(\mathfrak{g}_\Z)=\mathfrak{g}_\Z$ as in \cite{MP95}.
Let $\alpha\in\Delta_+^{\re}$, then $\alpha=w\alpha_{i}$ for some $w\in W$ and $i\in I$.
We define the \emph{root vectors} (\cite{MP95}, Section 4.1) by
\begin{align*}
x_{\a}:= \Ad(\widetilde{w})(e_i) \in\frak{g}_{\alpha}\cap \frak{g}_\Z \quad\text{and}\quad
x_{-\a}:=\Ad(\widetilde{w})( f_i) \in\frak{g}_{-\alpha}\cap \frak{g}_\Z,
\end{align*}
so that $x_{-\a}=-\omega(x_\a)$ and $[x_{\a},x_{-\a}] = \a^{\vee}$. 
In particular, $x_{\a_i}=e_i$ and $x_{-\a_i}=f_i$. The set $\{x_{\a},
 x_{-\a},\ \a^{\vee}\}$ 
forms a basis for a 
subalgebra isomorphic to $\frak{sl}_2$.

\begin{lemma}\label{L-wact}
Let $w\in W$, $\mu\in \wts(V)$, $\a\in \Delta_+^{\re}$, $m\in\N$.
\begin{enumerate}[\hspace{.5cm}(a) ]
\item $\widetilde{w}\cdot V_{\mu,\Z} = V_{w\mu,\Z}$.
\item\label{L-wact-onVlambda} $V_{w\lambda,\Z} = \widetilde{w}\cdot V_{\lambda,\Z} =\Z v_{w\lambda}$ where $v_{w\lambda}=\widetilde{w}\cdot v_\lambda$.
\vspace{1mm}
\item\label{L-wact-onx} $\dfrac{{x_{\pm\a}}^m}{m!}\in \mathcal{U}_\Z$ and $\dfrac{{x_{\pm\a}}^m}{m!}\cdot V_{\mu,\Z} \subseteq
V_{\mu\pm m\a,\Z}$.
\end{enumerate}
\end{lemma}
\begin{proof}
Lemma~3.8(a) in \cite{Kac90} implies $\widetilde{w_i}\cdot V_\mu=V_{w_i\mu}$ for $i\in I$, and so
$\widetilde{w}\cdot V_\mu=V_{w\mu}$ by induction.
Combining this with $\widetilde{w}\cdot V_\Z=V_\Z$ and $V=\bigoplus_{\mu\in\wts(V)} V_\mu$ we get (a).
Parts (b) and (c) are now straightforward.
\end{proof}

For $t\in\Q$, set 
\begin{eqnarray*}
\chi_{\alpha}(t):=\exp (t \rho(x_{\alpha}))=\sum^\infty_{m=0}\frac{t^m \rho\left(x_{\alpha}\right)^m}{m!}\in\GL(V).
\end{eqnarray*} 
We have 
$\chi_{\alpha}(t)=\widetilde{w}\chi_{\alpha_i}(t)\widetilde{w}^{-1}\in G$.
We define the (real) \emph{root group} and its $\Z$-subgroup as
$$U_{\alpha}=\{\chi_{\alpha}(t)\mid t\in \mathbb{Q}\}\quad\text{and}\quad U_{\alpha}(\mathbb{Z})=\{\chi_{\alpha}(t)\mid t\in\mathbb{Z}\}.$$ 
Given a set of positive real roots $\Omega\subseteq\Delta_+^{\re}$, we define
$$U_\Omega=\langle U_\a\mid\a\in\Omega\rangle \quad\text{and}\quad 
U_\Omega(\Z)=\langle U_\a(\Z)\mid\a\in\Omega\rangle .$$
In particular, the  \emph{unipotent subgroup} is $U:=U_{\Delta_+^{\re}}$ and its $\Z$-subgroup is
$U(\Z):=U_{\Delta_+^{\re}}(\Z)$.
It is clear that $U_\Omega(\Z)\subseteq U_\Omega\cap G(\Z)$.
The reverse inclusion is conjectured to be true for reasonable choices of $\Omega$ but will not be needed for the results in this paper.

We end this section with the following:
\begin{lemma}\label{hwtstab}
For all $u\in U$, $u\cdot v_{\lambda}=v_{\lambda}$.
\end{lemma}
\begin{proof}
It suffices to prove this statement for  a generator $\chi_{\alpha}(t)$ of $U_\alpha\subseteq U$.  Now
\begin{align*}
\chi_{\alpha}(t)\cdot v_{\lambda}
=v_{\lambda}+tx_{\alpha} \cdot v_{\lambda}+t^{2}\,\frac{{x_{\alpha}}^{2}}{2!} \cdot v_{\lambda}+\dots+t^{m}\,\frac{{x_{\alpha}}^{m}}{m!} \cdot v_{\lambda},
\end{align*}
for some $m$ since $x_{\alpha}$ acts nilpotently.
 But ${x_{\alpha}}^{i}\cdot v_{\lambda}\in V_{\lambda+i\alpha}=0$, since $\lambda +i\alpha$ is not a weight of $V$ for $i>0$. Thus,
$
\chi_{\alpha}(t) \cdot v_{\lambda}=v_{\lambda}.
$
\end{proof}

\section{Strong integrality of real root groups}\label{integsec6}
Throughout this section, fix $\alpha\in \Delta^{\re}$.
The main result of this section is to prove strong integrality of the real root group
$$U_{\alpha}=\{\chi_{\alpha}(t)\mid t\in \mathbb{Q}\}.$$
This provides the base case for Theorem~\ref{T-main}. 
Integrality of $U_\a$ will follow from the following lemma.
\begin{lemma}\label{Strong} Strong integrality implies integrality.  
\end{lemma}
\begin{proof} Let $M\subseteq G$ be strongly integral  with respect to $v_0\in V^{\lambda}$. 
Let $g\in M\cap\Gamma(\Z)$. Then $g \cdot v \in V^{\lambda}_{\mathbb{Z}}$ for all   $v \in V^{\lambda}_{\mathbb{Z}}$. In particular $g\cdot v_0 \in V^{\lambda}_{\mathbb{Z}}$ and so $g \in G(\Z)$ by strong integrality.
\end{proof}


Define
$$x_\a^{[m]} := \frac{{x_\a}^m}{m!}.$$
The following technical lemma is essential to our integrality results:
\begin{lemma}\label{L-xx}
Let $\mu$ be a weight such that $\mu+\alpha$ is not a weight of $V$. 
Further assume that $v_{\mu}\in V_{\mu}$ and $n:= \langle\mu,\a^\vee\rangle>0$.
Then
$$x_{\alpha}^{[n]} x_{-\alpha}^{[n]} \cdot v_{\mu} = v_\mu.$$
\end{lemma}
\begin{proof}
First we prove that
$$[x_{\alpha},{x_{-\alpha}}^{k}]=k{x_{-\alpha}}^{k-1}\left(\alpha^{\vee}-(k-1)\right)$$
by induction on $k\in\N$:
\begin{align*}
[x_\a,{x_{-\a}}^k] &= x_\a {x_{-\a}}^k -{x_{-\a}}^k x_\a \\
&=x_\a{x_{-\a}}^k - {x_{-\a}}^{k-1}x_\a{x_{-\a}} +{x_{-\a}}^{k-1}x_\a{x_{-\a}}   -{x_{-\a}}^k x_\a\\
&= [x_\a,{x_{-\a}}^{k-1}] {x_{-\a}} + {x_{-\a}}^{k-1} [x_\a,{x_{-\a}}]\\
&=  (k-1) {x_{-\alpha}}^{k-2}\left(\alpha^{\vee}-k+2 \right){x_{-\a}} 
+ {x_{-\alpha}}^{k-1} \alpha^{\vee}\\
&= (k-1) {x_{-\alpha}}^{k-2}(\alpha^{\vee}x_{-\a}) -(k-1)(k-2){x_{-\alpha}}^{k-1} + {x_{-\alpha}}^{k-1} m\alpha^{\vee}\\
&= (k-1) {x_{-\alpha}}^{k-2}(x_{-\a}\alpha^\vee-2x_{-\a}) -(k-1)(k-2){x_{-\alpha}}^{k-1} + {x_{-\alpha}}^{k-1} \alpha^{\vee}\\
&= k {x_{-\alpha}}^{k-1}\alpha^\vee -k(k-1){x_{-\alpha}}^{k-1}.
\end{align*}

Note that $x_\a\cdot v_\mu\in V_{\mu+\a}=0$, so $x_\a\cdot v_\mu=0$.
Also $\a^\vee \cdot v_\mu=\langle\mu,\alpha^{\vee}\rangle v_\mu=nv_\mu$.
Next we prove
\begin{eqnarray*}
{x_{\alpha}}^{m}{x_{-\alpha}}^{m}\cdot v_{\mu}=m!\prod_{j=0}^{m-1}(n-j)\cdot v_{\mu}.
\end{eqnarray*}
by induction on $m$:
\begin{align*}
{x_{\alpha}}^{m}{x_{-\alpha}}^{m}\cdot  v_{\mu}
&={x_\a}^{m-1}\left([x_{\alpha}, {x_{-\alpha}}^{m}]+{x_{-\alpha}}^{m}x_{\alpha}\right)\cdot  v_{\mu}\\
&= {x_\a}^{m-1}[x_{\alpha}, {x_{-\alpha}}^{m}]\cdot  v_{\mu}\\
&= {x_{\alpha}}^{m-1}{x_{-\alpha}}^{m-1} m\left(\a^\vee-(m-1)\right)\cdot  v_{\mu}\\
&= m\left(n-(m-1)\right) {x_{\alpha}}^{m-1}{x_{-\alpha}}^{m-1}\cdot  v_{\mu}\\
&= m!\prod_{j=0}^{m-1}(n-j)\cdot v_{\mu}.
\end{align*}
Now take $m=n$ to get
${x_{\alpha}}^{n}{x_{-\alpha}}^{n}\cdot  v_{\mu} = n!\,n!\,v_\mu$, so $x_{\alpha}^{[n]} x_{-\alpha}^{[n]} \cdot v_{\mu} = v_\mu$.
\end{proof}

\begin{lemma}[Lemma~7.2 of \cite{St}]\label{hwt20}
Let $\a\in \Delta^{\re}$ and $t\in\Q$.
If $\mu$ is a weight of $V$ and $v\in V_{\mu}$, then
\[
\chi_\alpha(t)\cdot v=v+\sum_{m\in \mathbb{N}} t^m (x_\a^{[m]}\cdot v),
\]
where $x_\a^{[m]}\cdot v \in V_{\mu+m\alpha}$, and only finitely many of the terms $x_\a^{[m]}\cdot v$ are non-zero.
If moreover $v\in V_{\mu,\Z}$, then $x_\a^{[m]}\cdot v \in V_{\mu+m\alpha,\Z}$.
\end{lemma}
\begin{proof}
If $v\in V_{\mu, \Z}$, 
then $x_\a^{[m]}\cdot v\in V_{\mu+m\alpha,\Z}$ by Lemma~\ref{L-wact}(\ref{L-wact-onx}).
\end{proof}


\begin{proposition}\label{P-onetermstrong}
Let $\alpha\in\Delta^{\re}$.
If $\chi_\a(t) \widetilde{w}_\a\cdot  v_{\lambda}\in V_{\mathbb{Z}}$, then $\chi_\a(t)\in U_\a({\mathbb{Z}})$.
\end{proposition}
\begin{proof}
Let $\a$ be positive. Suppose $\chi_\alpha(t)\cdot v_0\in V_\Z$ where  $v_0:=\widetilde{w}_\a\cdot v_\lambda$.
Let  $n=\langle \lambda, \alpha^{\vee}\rangle$ so that $w_\a\lambda =\lambda-n\a$.
Note that $n>0$ since $\lambda$ is regular.
Now $V_{\lambda-n\a,\Z}= \Z v_0$ by Lemma~\ref{L-wact}(\ref{L-wact-onVlambda}).
But $x_{-\a}^{[n]}\cdot v_\lambda$ is also in $V_{\lambda-n\a,\Z}$
by Lemma~\ref{L-wact}(\ref{L-wact-onx}), so $x_{-\a}^{[n]}\cdot v_\lambda = mv_0$
for some integer $m$.
We have
$$
\chi_\alpha(t) \cdot v_0= v_0+tx_{\alpha}\cdot v_0+t^2x_{\a}^{[2]}\cdot v_0 +\dots+ t^nx_{\a}^{[n]}\cdot v_0.
$$

Now consider the $V_{\lambda,\Z}$ component:
$$t^nx_{\a}^{[n]}\cdot v_0 = \frac{t^n}{m} x_{\a}^{[n]}x_{-\a}^{[n]}\cdot v_\lambda
= \frac{t^n}{m} v_\lambda$$
by Lemma~\ref{L-xx}.
Hence
$\frac{t^n}{m}$ is an integer, so $t^n$ is an integer, and so $t$ is an integer.
The proof for negative $\a$ is similar.
\end{proof}

\newpage
Since $U_\a(\Z)\subseteq U_\a\cap G(\Z)$, we get:
\begin{cor}\label{C-onetermstrong}\label{C-oneterm}
If $\alpha\in\Delta^{\re}$, then $U_\a$ 
is strongly integral with respect to $\widetilde{w}_\a\cdot v_\lambda$.
Hence $U_\a$ is integral.
\end{cor}

It follows immediately that for $\alpha\in\Delta^{\re}$ and $u\in U(\Z)$, if $\chi_{\alpha}(s) u\in \Gamma(\Z)$ then $s\in\Z$. 


\section{Inversion sets and subgroups}

\subsection{Inversion sets and orderings}\label{SS-inv}
Let $w\in W$. We define the corresponding \emph{inversion set} of roots by
\begin{eqnarray*}
 \Phi_{(w)}&=&\{\beta\in \Delta_{+}\mid w^{-1}\beta\in \Delta_{-}\}=\Delta_{+}\cap w(\Delta_{-}).
 \end{eqnarray*}
We note the following standard properties of  $\Phi_{(w)}$, which can be proven by induction on $\ell(w)$.
\begin{lemma}\label{Invsets}
Suppose that $w\in W$ has the reduced expression $w_{i_1}w_{i_2}\cdots w_{i_k}$.
\begin{enumerate}[\hspace{.5cm}(a) ]
\item $\Phi_{(w)}=\{\alpha_{i_1}, \;w_{i_1}\alpha_{i_2}, \;w_{i_1}w_{i_2}\alpha_{i_3},\;\dots,\; w_{i_1}w_{i_2}\cdots w_{i_{k-2}}w_{i_{k-1}}\alpha_{k}\}\subseteq\Delta_{+}^{\re}$.
\item $\Phi_{(w)}$ has cardinality $k=\ell(w)$. 
\item If $w'=w_{i_1}w_{i_2}\cdots w_{i_{k-1}}$, then $\Phi_{(w)}=\Phi_{(w')}\sqcup \{w'\alpha_{i_k}\}.$
\item\label{Invsets-decomp} If ${w}''=w_{i_2}\cdots w_{i_k}$, then $\Phi_{(w)}=\{\alpha_{i_1}\}\sqcup {w_{i_1}}\Phi_{({w''})}$.
\item For $\mu\in \mathfrak{h}^{\ast}$,
$w\mu=\mu-\langle \mu,\alpha_{i_{1}}\rangle \alpha_{i_{1}}-\langle \mu,\alpha_{i_{2}}\rangle w_{i_{1}}\alpha_{i_{2}}-\dots-\langle \mu,\alpha_{i_{k}}\rangle w_{i_{1}}w_{i_{2}}\cdots w_{i_{k-1}}\alpha_{i_{k}}.$
\end{enumerate}
\end{lemma}

 We also note the immediate consequences of Lemma~\ref{Invsets}:
 \begin{lemma}\label{invdec}$\;$
\begin{enumerate}[\hspace{.5cm}(a) ]
\item For $\alpha\in \Delta_{+}^{\re}$, there exists $w\in W$ such that $\alpha\in \Phi_{(w)}$.
\item For a finite subset $\Omega\subseteq  \Delta^{\re}_{+}$, there exists a finite set $T\subseteq W$ such that
 $\Omega\subseteq\bigcup_{w\in T}\Phi_{w}.$
\end{enumerate}
\end{lemma} 

We equip $\Phi_{(w)}$ with an ordering
$$\alpha_{i_1} \prec w_{i_1}\alpha_{i_2} \prec w_{i_1}w_{i_2}\alpha_{i_3}\prec\cdots\prec w_{i_1}w_{i_2}\cdots w_{i_{k-2}}w_{i_{k-1}}\alpha_{k},$$
as described by Papi in \cite{P}. Note that this ordering depends of the choice of reduced word for~$w$.

\subsection{Inversion subgroups}\label{invgrp}
For $w\in W$, the {\it inversion subgroup}  is defined as 
$$U_{(w)}=U_{\Phi_{(w)}}=\langle U_\a \mid \a\in\Phi_{(w)}\rangle.$$
\begin{lemma}\label{lemdec2} 
Each element $u_{ (w)}\in U_{(w)}$ can be written uniquely as 
$$u_{ (w)}=\prod_{\gamma\in \Phi_{w} }u_{\b},$$
where $u_\b\in U_\b$ and the product is in the ordering on $\Phi_{(w)}$.
\end{lemma}
\begin{cor}
$U_{(w)}(\Z) = U_{(w)}\cap U(\Z)$. 
\end{cor}

For $w\in W$, set $U^{(w)}=U\cap w^{-1}Uw$. We end this section with the following lemma from Subsection~4.2.6 of  \cite{PP}.
\begin{lemma}\label{conlem1}
$
U=U^{(w)}U_{(w)}=U_{(w)}U^{(w)}.
$
\end{lemma}

\section{Integrality of inversion subgroups}\label{S-proof}

To prove Theorem~\ref{T-main}, we need a preliminary result, which we discuss next.
Define
$$x_{\b_1,\b_2,\dots,\b_k}^{[i_1,i_2,\dots,i_k]} := x_{\b_1}^{[i_1]} x_{\b_2}^{[i_2]} \cdots x_{\b_k}^{[i_k]}.$$
 
%
 The proof of the following is similar to Lemma~\ref{hwt20}:
\begin{lemma}\label{hwt2}
Let $\beta_{i}\in \Delta^{\re}_{+}$ and $t_{i}\in\mathbb{Q}$ for $i=1,\dots,k$.
If $\mu$ is a weight of $V$ and $v\in V_{\mu}$, then
\[
\mathlarger{\prod}_{i=1}^{k}\chi_{\beta_{i}}(t_{i})\cdot v=\mathlarger{\mathlarger{\sum}}_{i_{1},i_{2},\dots, i_{k}\in \mathbb{Z}_{\ge 0}}  t_{1}^{i_{1}}t_{2}^{i_{2}}\cdots t_{k}^{i_{k}} \left( x_{\b_1,\b_2,\dots,\b_k}^{[i_1,i_2,\dots,i_k]} \cdot v\right)
\]
where $x_{\b_1,\b_2,\dots,\b_k}^{[i_1,i_2,\dots,i_k]} \cdot v \in V_{\mu+i_{1}\beta_{1}+i_{2}\beta_{2}+i_{k}\beta_{k}}$ and only finitely of the terms $x_{\b_1,\b_2,\dots,\b_k}^{[i_1,i_2,\dots,i_k]} \cdot v$ are non-zero.
If moreover $v\in V_{\mu,\mathbb{Z}}$, then $x_{\b_1,\b_2,\dots,\b_k}^{[i_1,i_2,\dots,i_k]} \cdot v \in V_{\mu+i_{1}\beta_{1}+i_{2}\beta_{2}+i_{k}\beta_{k},\mathbb{Z}}$.
\end{lemma}

\begin{proposition}\label{P-main}
Let $w\in W$ and
$u_{(w)}\in U_{(w)}$. If $u_{(w)}  \widetilde{w}\cdot  v_{\lambda}\in V_{\mathbb{Z}}$, then $u_{(w)}\in U_{(w)}({\mathbb{Z}})$.
\end{proposition}
Theorem~\ref{T-main} follows immediately from this result since $U_{(w)}(\Z)\subseteq U_{(w)}\cap G(\Z)$.

\begin{proof}
We proceed by induction on $k=\ell(w)$. The base case $\ell(w)=1$ is just Proposition~\ref{P-onetermstrong}
for a simple root $\a_i$.
So assume $\ell(w)\ge 2$ and the inductive hypothesis. 
Let $\Phi_{(w)}=\{\b_1, \dots \b_k\}$ in the usual order.
So $\b:=\b_1$ is a simple root and $w=w_{\b}w''$ with $\ell(w'')=k-1$,
and $\Phi_{(w'')} =\{w_\b \b_2,\dots, w_\b \b_k\}$ by Lemma~\ref{Invsets}(\ref{Invsets-decomp}).
In particular, $w_\b\b_2,\dots,w_\b \b_k$ are positive roots.
Write
$$u_{(w)} = \prod_{i=1}^k \chi_{\b_i}(t_i) = \chi_\b(t) \prod_{i=2}^k \chi_{\b_i}(t_i),$$
for $t=t_1,t_2,\dots,t_k\in\Q$.

Write $\mu=w''\lambda$ and $v_\mu=\widetilde{w}''\cdot v_\lambda$.
Let $n=  \langle \mu,\b^\vee\rangle$, so that $w\lambda=w_{\b}\mu =\mu-n\b$.
Let $v_{\mu-n\b}=\widetilde{w}_\b\cdot v_\mu = \widetilde{w}\cdot v_\lambda$.
Then
\begin{align*}
u_{(w)}\widetilde{w} \cdot v_\lambda 
&=u_{(w)} \cdot v_{\mu-n\b}
=  \mathlarger{\mathlarger{\sum}}_{\substack{{i_{1},i_{2},\dots, i_k}
\in \mathbb{Z}_{\ge 0}}}  t_1^{i_1}t_2^{i_2}\cdots t_k^{i_k} 
\left( x^{[i_2,\dots,i_k]}_{\b_2,\dots, \b_k}  \cdot v_{\mu-n\b} \right)\\
   &= t^n x_{\b}^{[n]} \cdot v_{\mu-n\b} +
\mathlarger{\mathlarger{\sum}}_{\substack{{i\in\Z_{\le n},\;i_{2},\dots, i_{k}\in \mathbb{Z}_{\ge 0}}\\\text{at least one non-zero}}}  
 t^{n-i} t_2^{i_2}\cdots t_k^{i_k}  \left( x^{[n-i,i_2,\dots,i_k]}_{\b,\b_2,\dots, \b_k}  \cdot v_{\mu-n\b} \right),
\end{align*}
where we have substituted $i=n-i_1$ and extracted the term with $i,i_2,\dots,i_k=0$ from the sum.
Now $t^n x_{\b}^{[n]} \cdot v_{\mu-n\b}\in V_\mu$ and $x^{[n-i,i_2,\dots,i_k]}_{\b,\b_2,\dots, \b_k}  \cdot v_{\mu-n\b}\in 
V_\nu$ for $\nu=\mu-i\b+\sum_{j=2}^k i_j\b_j$.
Now 
\begin{align*}
  (w'')^{-1}(\nu) &= \lambda - (w'')^{-1}(i\b) +(w'')^{-1}\left(\sum_{j=2}^k i_j\b_j\right)
\ge \lambda - i(w'')^{-1}(\b),
\end{align*}
but $(w'')^{-1}(\b) > 0$, so we must have $i\ge 0$ for $\nu$ to be a weight of $V$.
So we can take $i$ to be in the range $0\le i\le n$.
If $\mu=\nu$, then $i\b=\sum_{j=2}^k i_j\b_j$,
but $w_\b\left( i_j\b_j\right)\ge0$ and $w_\b(i\b)\le0$, so  $i,i_1,\dots, i_k=0$.
Hence the $V_{\mu,\Z}$ component of $u_{(w)}\widetilde{w}\cdot v_\lambda\in V_\Z$ is just
$ t^n x_{\b}^{[n]} \cdot v_{\mu-n\b}$.

Now $V_{\mu,\Z}=\Z v_\mu$, so
$t^n x_{\b}^{[n]} \cdot v_{\mu-n\b}=m v_{\mu}$
for some integer $m$.
Also 
$
 x_{-\b}^{[n]} \cdot v_\mu\in V_{\mu-n\b,\Z} =
\Z v_{\mu-n\b},
$
so 
$x_{-\b}^{[n]}\cdot  v_{\mu}=m'v_{\mu-n\b}$ for some integer $m'$.
Hence
$$
mm'v_{\mu-n\b} =   m x_{-\b}^{[n]}\cdot v_{\mu} =
t^n x_{-\b}^{[n]} x_{\b}^{[n]} \cdot v_{\mu-n\b} = t^n v_{\mu-n\b}
$$
by Lemma~\ref{L-xx} for $\a=-\b$, since $\mu-n\b-\b\notin\wts(V)$ as 
$w^{-1}(\mu-n\b-\b) = w^{-1}(w\lambda -\b)= \lambda - w^{-1} \b >\lambda$.
Hence  $t^n=mm'$ is an integer, so $t$ is an integer.

Now 
$$
u_{(w)}\widetilde{w}=\chi_\b(t) \prod_{i=2}^k \chi_{\b_i}(t_i) \,\widetilde{w}_\b \widetilde{w}''
=\chi_\b(t)\widetilde{w}_\b u_{(w'')} \widetilde{w}'',
$$
where
$ u_{(w'')}= \prod_{i=2}^k \chi_{w_\b \b_i}(t_i)\in U_{(w'')}$.
Now  $\chi_\b(t)  \widetilde{w}_\b\in G(\Z)$, so
$$u_{(w'')} \widetilde{w}''\cdot v_\lambda =\left( \chi_\b(t)  \widetilde{w}_\b\right)^{-1}\cdot \left(u_{(w)} \widetilde{w}\cdot v_\lambda\right) \in V_\Z,$$
and so $u_{(w'')}\in U_{(w'')}(\Z)$ by induction. Hence $u_{(w)}= \chi_\b(t)\widetilde{w}_\b u_{(w'')} \widetilde{w}_\b^{-1}\in U_{(w)}(\Z)$.
\end{proof}

\section{Further integrality results}

\subsection{Integrality of finite dimensional unipotent groups}\label{SS0-findim}

For this section, we assume that $A$ has finite type. In this case, $A$ is a Cartan matrix and $\frak g$ is a finite dimensional semi-simple Lie algebra over $\Q$. The group construction in Section~\ref{S-KMring} can be carried out with the same external data: a highest weight representation $V$ with highest weight $\lambda$, a $\Z$-form of the universal enveloping algebra (constructed by Cartier and Kostant, see \cite{Kostant1966}) and a lattice $V_{\Z}$ in $V$ (called an {\it admissible lattice} in \cite{Chev1995}).

The group $G(\Q)$ is then a semi-simple algebraic group and $G(\Q)\cong G_\Z(\Q)$, where $G_\Z(\Q)$ is Chevalley's group  scheme, now known as the Chevalley--Demazure group scheme (\cite{Chev2005}).

Our construction of the group $G(\Q)$ does not explicitly involve a lattice $L$ between the root lattice and the weight lattice, as in Chevalley's construction. With the assumption that the set of weights of $V$ contains all the fundamental weights, our group $G(\Q)$ coincides with Chevalley's simply connected group $ G_\Z(\Q)$.

As a corollary of Theorem~\ref{T-main}, we have the following.

 \begin{cor}\label{findimcor}
Suppose that the matrix $A$ has finite type. Suppose also that the set of weights of $V$ contains all the fundamental weights. Then  the  unipotent subgroup $U$ of $G(\Q)$ is integral.
\end{cor}

\begin{proof}  Since $A$ has finite type, the Weyl group $W$ is finite. Moreover, $W$ contains the longest element, denoted $w_{0}$, which flips all positive roots to negative roots. That is $\Phi_{w_{0}}=\Delta^{+}$, which gives $U_{w_{0}}=U$. Combining this with Theorem~\ref{T-main} and Lemma~\ref{Strong},
it follows that $U$ is integral.
\end{proof}

\subsection{Integrality of groups generated by commutating real root subgroups}\label{SS-abel}
We can now prove the result from Section~\ref{S-intro} that gives integrality of subgroups of $U$ generated by commuting real root groups.

\begin{abeltheorem}
Let $U$ be the positive unipotent subgroup of $G=G^\lambda(\Q)$. 
If $M$ is a subgroup of $U$ generated by commuting real root subgroups, then
$M$ is integral.
\end{abeltheorem}
\begin{proof}
Assume, for the sake of contradiction, that
\begin{eqnarray*}\label{cusp1a}
u=\chi_{\beta_{1}}(t_{\beta_{1}})\chi_{\beta_{2}}(t_{\beta_{2}})\cdots \chi_{\beta_{N}}(t_{\beta_{N}})\in M
\end{eqnarray*} 
is the shortest word such that 
$u\cdot V_{\mathbb{Z}} = V_{\mathbb{Z}}$ but $u\notin U(\Z)$.
By Corollary~\ref{C-oneterm}, we must have $N>1$.
Now $\Omega:=\{\beta_{1},\beta_{2},\dots, \beta_{N}\}\subseteq \Delta^{\re}_{+}$ is a finite set of real roots whose corresponding root subgroups commute with each other.
By part (b) of Lemma~\ref{invdec}, there exists $w\in W$ such that $\Omega\cap\Phi_{(w)}\neq \emptyset$.
Using the commutating of the real root groups $U_{\beta_{j}}$, we can rearrange the product above as
 $$u=u_{(w)}u^{(w)}= u^{(w)}u_{(w)},$$
 where $u_{(w)} = \prod_{\beta\in\Omega\cap\Phi_{(w)}} \chi_\beta(t_\beta)$ and
$u^{(w)} = \prod_{\beta\in\Omega\setminus\Phi_{(w)}} \chi_\beta(t_\beta)$.

Now we have
\begin{eqnarray*}
u\widetilde{w}\cdot v_{\lambda}=u_{(w)}u^{(w)}\widetilde{w}\cdot v_{\lambda}
=u_{(w)}\widetilde{w}\cdot (\widetilde{w}^{-1}u^{(w)}\widetilde{w} \cdot v_{\lambda}).
\end{eqnarray*}  
But $\widetilde{w}^{-1}u^{(w)}\widetilde{w}\in U$, so 
$\widetilde{w}^{-1}u^{(w)}\widetilde{w} \cdot v_{\lambda}=v_\lambda$
by Lemma~\ref {hwtstab}. 
So  $u_{(w)}\widetilde{w}\cdot v_{\lambda}=u\widetilde{w}\cdot v_{\lambda}\in V_{\mathbb{Z}}$.
But $U_{(w)}$ is strongly integral with respect to $\widetilde{w}\cdot v_{\lambda} $
by Theorem~\ref{T-main}, so $u_{(w)}\in U (\mathbb{Z})$.

Hence ${u_{(w)}}^{-1}\in U (\mathbb{Z})$ and so
$$u^{(w)}\cdot V_{\mathbb{Z}} = u \cdot({u_{(w)}}^{-1} \cdot V_{\mathbb{Z}}) 
= u\cdot V_{\mathbb{Z}} = V_{\mathbb{Z}}.$$
Since $\Omega\cap\Phi_{(w)}$ is nonempty, $u^{(w)}= \prod_{\beta\in\Omega\setminus\Phi_{(w)}} \chi_\beta(t_\beta)$ is a shorter word than $u$ and so our assumption implies $u^{(w)}\in U (\mathbb{Z})$. 
But now
$u=u_{(w)}u^{(w)}\in U (\mathbb{Z})$, which is a contradiction.
\end{proof}

\subsection{Integrality of subgroups of $U$ in the rank $2$ Kac--Moody case}\label{SS-rank2}
We start with a method for  constructing  infinite-dimensional integral unipotent subgroups.
\begin{proposition}\label{P-addunion}
Let $\Omega=\bigcup_{n=1}^\infty \Omega_n\subseteq\Delta_+^{\re}$ where $\Omega_1\subseteq\Omega_2\subseteq\cdots$ and
assume that each $U_{\Omega_n}$ is integral.
Then  $U_{\Omega}=\langle U_\a \mid \a\in\Omega\rangle$ is integral. 
\end{proposition}
\begin{proof}
Take $u\in U_{\Omega}$ with
$u\cdot V_\Z= V_\Z$. Then
$u=\prod_{i=1}^N \chi_{\beta_i}(a_i)$
for some $\beta_i\in\Omega$, $a_i\in\Q$. But now the finite set of roots $\{\beta_1,\dots,\beta_N\}$ must be contained
in $\Omega_n$ for some $n$. 
Hence $u\in U_{\Omega_n}$, which is an integral subgroup. So $u\in G(\Z)$.
\end{proof}

One way to construct such sets is by taking a sequence of simple reflections $w_{i_1}, w_{i_2},\dots$ such that every initial subsequence gives a reduced word $w_{i_1}\cdots w_{i_n}\in W$. Then
each group $U_{(w_{i_1}\cdots w_{i_n})}$ is integral by Theorem~\ref{T-main}. Now take
\begin{align*}
\Omega &:= \bigcup_{n=1}^\infty \Phi_{(w_{i_1}\cdots w_{i_n})}
 = \{\alpha_{i_1},\; w_{i_1}\alpha_{i_2}, \;w_{i_1}w_{i_2}\alpha_{i_3},\;\dots,\; w_{i_1}w_{i_2}\cdots w_{i_{k-2}}w_{i_{k-1}}\alpha_{k},\;\dots\}.
 \end{align*}
Then the subgroup $U_{\Omega}$ is integral.

Now suppose that $\mathfrak{g}$ has rank 2. 
A detailed description of the structure of $\Delta^{\re}$ and $U$ in rank 2 can be found in~\cite{CKMS}. 
We define
\begin{align*}
 \Omega_1 :=  \{ \alpha _{1},\ w_{1}\alpha _{2},\ w_{1}w_{2} \alpha_{1},\;\dots\}\quad\text{and}\quad
 \Omega_2 :=  \{ \alpha _{2},\ w_{2}\alpha _{1},\ w_{2}w_{1} \alpha_{2},\;\dots\}.
\end{align*}
In fact, $\Omega_1$ (resp.\ $\Omega_2$) is the set of all positive real roots on the lower (resp.\ upper)
branches of the hyperbolas defined in~\cite{CKMS}.

In this case $U=U_{1}{\Large{*}}\,U_{2}$, where for $i=1,2$,  the groups $U_i:= U_{\Omega_i}$ are abelian if $A$ is symmetric, and nilpotent of class 2 if $A$ is not symmetric (see \cite{CKMS}). 

The following theorem follows immediately from Proposition~\ref{P-addunion}.
\begin{rk2theorem}
Suppose that $G(\Q)$ has rank 2 with $U=U_{1}{\Large{*}}\,U_{2}$, where $U_i= U_{\Omega_i}$.   Then the groups $U_i$ are integral.
\end{rk2theorem}


\bibliographystyle{amsalpha}
\bibliography{GPmath}{}

\end{document}